\pdfoutput=1
\documentclass[12pt, reqno]{amsart}

\usepackage{amsmath,amsfonts,amssymb}
\usepackage{mathtools}
\usepackage[utf8]{inputenc}
\usepackage[english]{babel}

\usepackage{a4wide}
\usepackage{cite}

\DeclareUnicodeCharacter{FB01}{fi}
\newtheorem{theorem}{Theorem}[section]

\newtheorem{remark}[theorem]{Remark}

\newtheorem*{aim-non}{Aim}

\newtheorem*{conjecture-non}{Conjecture}

\theoremstyle{definition}
\newtheorem{definition}{Definition}

\begin{document}
	\title[Spectral Data Parabolic Projective Symplectic/Orthogonal Higgs]
	{Spectral Data For Parabolic Projective Symplectic/Orthogonal Higgs Bundles}
	
	\author{Sumit Roy}\
	\address{Center for Geometry and Physics, Institute for Basic Science (IBS), Pohang 37673, Korea}
\email{sumit@ibs.re.kr}
\thanks{E-mail : sumit@ibs.re.kr}
\thanks{Affiliation: Center for Geometry and Physics, Institute for Basic Science (IBS), Pohang 37673, Korea}
	
	\begin{abstract}
	 Hitchin in [Duke Math. J. 54 (1), 91-114 (1987)] introduced a proper morphism from the moduli space of stable $G$-Higgs bundles ($G=\mathrm{GL}(n,\mathbb{C}),\mathrm{Sp}(2m,\mathbb{C})$ and $\mathrm{SO}(n,\mathbb{C})$) over a curve to a vector space of invariant polynomials and he described the generic fibers of that morphism. In this paper, we first describe the generic Hitchin fibers for the moduli space of stable parabolic projective symplectic/orthogonal Higgs bundles without fixing the determinant. We also describe the generic fibers when the determinant is trivial.
	    
	\end{abstract}
	\maketitle
	
	\section{Introduction}
	Let $X$ be a compact (closed and connected) Riemann surface $X$ of genus $g \geq 2$. Higgs bundles over $X$ were introduced by Hitchin in \cite{H87a}. A \textit{Higgs bundle} over $X$ is a pair $(E, \phi)$ consisting of a a holomorphic vector bundle $E$ and a Higgs field $\phi : E \to E \otimes K$, where $K$ is the canonical bundle over $X$. The coefficients of the characteristic polynomial of $\phi$ defines a morphism
	$h : \mathcal{M}_{\mathrm{Higgs}}(r,d) \longrightarrow \mathcal{A}\coloneqq \bigoplus_{i=1}^{r} H^0(X, K^i)$,
	from the moduli space of stable Higgs bundles over $X$ of fixed rank $r$ and degree $d$ to a vector space $\mathcal{A}$, called the \textit{Hitchin map} (see \cite{H87}). Hitchin in \cite{H87}, showed that the generic fibers of $h$ are abelian varieties and this map gives the Higgs bundles moduli space a structure of an algebraically completely integrable system. Later in \cite{M94}, Markman generalized this result for the moduli space of $L$-twisted Higgs bundles $(E,\phi_L)$, where $L$ is a line bundle over $X$ and $\phi_L : E \to E \otimes L$.
	
	Let $D \subset X$ be a fixed finite subset. The notion of parabolic bundles over a curve and their moduli spaces were constructed by Mehta and Seshadri in \cite{MS80}. Their motivation was to extend the Narasimhan-Seshadri correspondence in the case of irreducible unitary representations of $\pi_1(X-D)$. A \textit{parabolic bundle} is a holomorphic vector bundle together with a weighted flag over each parabolic point $p\in D$. A \textit{parabolic Higgs bundle} on $X$ is a parabolic bundle $E$ on $X$ together with a parabolic Higgs field $\phi : E \to E \otimes K(D)$. The moduli space of parabolic Higgs bundles was constructed by Yokogawa \cite{Y93}.
	
	Symplectic (resp. orthogonal) parabolic bundles are parabolic bundles with a suitably defined nondegenerate anti-symmetric (resp. symmetric) form taking values in a line bundle $L$ (see \cite{BMW11} for more details). In \cite{BR89}, Bhosle and Ramanathan described the notion of parabolic principal $G$-bundles, where $G$ is a connected reductive group, and also constructed its moduli space. When all weights are rational, the notion of symplectic (resp. orthogonal) parabolic bundles coincides with the notion of parabolic principal $G$-bundles where $G$ is a symplectic (resp. orthogonal) complex group (see \cite{BMW11}).  A symplectic (resp. orthogonal) parabolic Higgs bundle is a symplectic (resp. orthogonal) parabolic bundle together with a parabolic Higgs field which is compatible (in a suitable sense) with the symplectic (resp. orthogonal) structures.
	
	In \cite{H87}, Hitchin also showed that the moduli space of stable symplectic/orthogonal Higgs bundles also forms an algebraically completely integrable system, fibered over a vector space, either by a Jacobian or a Prym variety of so-called spectral curves. In \cite{R20}, Roy generalized this result for the moduli space of stable parabolic symplectic/orthogonal Higgs bundles. 
	
	In this paper, we consider the moduli space of parabolic projective symplectic/orthogonal Higgs bundles with fixed rank and degre and fixed parabolic structure. We know that for odd rank the projective orthogonal group is same as the odd orthogonal group, i.e.  $\mathrm{PSO}(2m+1,\mathbb{C}) =\mathrm{SO}(2m+1,\mathbb{C})$. Therefore, we only consider the projective symplectic group $\mathrm{PSp}(2m,\mathbb{C})$ and projective even orthogonal group $\mathrm{PSO}(2m,\mathbb{C})$. A parabolic $\mathrm{PSp}(2m,\mathbb{C})$-Higgs (resp. $\mathrm{PSO}(2m,\mathbb{C})$-Higgs) bundle lifts to a parabolic $\mathrm{GSp}(2m,\mathbb{C})$-Higgs (resp. $\mathrm{GSO}(2m,\mathbb{C})$-Higgs) bundle. We gave an alternative description of the Prym varieties and we consider an action of the Jacobian group $\mathrm{Jac}(X)$ on the Prym varieties. We showed that the generic Hitchin fibers are isomorpic to the quotient variety (see Theorem \ref{PSp} and Theorem \ref{PSO_theorem}).
	
	Finally we consider the case when the symplectic/orthogonal form takes values in the trivial line bundle $\mathcal{O}_X$ and we fix this line bundle for the respective moduli spaces. In this case, the $2$-torsion subgroup $\mathrm{Jac}_2(X)\subset \mathrm{Jac}(X)$ acts on the Prym varieties and the generic fibers of the Hitchin map are isomorphic to the quotient variety (see Theorem \ref{trivial_PSp} and Theorem \ref{trivial_PSO}).

	\section{Preliminaries}
	\subsection{Parabolic bundles}\label{parabolic} 
	Let $X$ be a compact Riemann surface of genus of genus $g\geq 2$. Fix a subset $D = \{p_1,\dots , p_n\} \subset X$ of $n$ distinct marked points. 
	\begin{definition}
A \textit{parabolic bundle} $E_*$ of rank $r$ over $X$ is a vector bundle $E$ of rank $r$ over $X$ with a parabolic structure over the subset $D$, i.e. for each point $p \in D$
	\begin{enumerate}
		\item every fiber $E_p$ has a filtration of subspaces, i.e.
		\[
		E_p \eqqcolon E_{p,1}\supsetneq E_{p,2} \supsetneq \dots \supsetneq E_{p,r_p} \supsetneq E_{p,r_p+1} =\{0\},
		\]
		\item an increasing sequence of real numbers (parabolic weights) satisfying 
		\[
		0\leq \alpha_1(p) < \alpha_2(p) < \dots < \alpha_{r_p}(p) < 1,
		\]
	\end{enumerate}
	where $1\leq r_p \leq r$ is an integer. 
\end{definition}
 We denote the collection of all parabolic weights by $\alpha =\{(\alpha_1(p),\alpha_2(p),\dots ,\alpha_{r_p}(p))\}_{p\in D}$ corresponding to a fixed parabolic structure. The parabolic structure $\alpha$ is said to have \textit{full flags} if 
\[\mathrm{dim}(E_{p,i}/E_{p,i+1}) = 1
\] 
for all $i \in \{1,\dots, r_p\}$ and for each $p\in D$, or equivalently $r_p=r$ for each $p\in D$. In this paper, we will assume that the parabolic structure have full flag at every parabolic point $p \in D$.

The \textit{parabolic degree} of $E_*$ is defined by
\[
\operatorname{pardeg}(E_*) \coloneqq \deg(E)+ \sum\limits_{p\in D}\sum\limits_{i=1}^{r_p} \alpha_i(p) \cdot \dim(E_{p,i}/E_{p,i+1})
\]
and the \textit{parabolic slope} of $E_*$ is defined by
\[
\mu_{\mathrm{par}}(E_*) \coloneqq \frac{\text{pardeg}(E_*)}{r}.
\]

\begin{definition}
A \textit{parabolic homomorphism} $\phi : E_* \to E^\prime_*$ between two parabolic bundles is a homomorphism between underlying vector bundles such that at each parabolic point $p \in D$ we have 
\[
\alpha_i(p) > \alpha_j^\prime(p) \implies \phi(E_{p,i}) \subseteq E_{p,j+1}^\prime.
\]
Furthermore, we call such a homomorphism \textit{strongly parabolic} if
\[ \alpha_i(p) \geq \alpha_j^\prime(p) \implies \phi(E_{p,i}) \subseteq E_{p,j+1}^\prime
\]
for every $p \in D$.
\end{definition}
We denote by $\mathrm{PEnd}(E_*)$ and $\mathrm{SPEnd}(E_*)$ the parabolic and strongly parabolic endomorphisms of $E_*$ respectively.

The \textit{dual} and \textit{tensor product} of parabolic bundles can be defined in a natural way (see \cite{Y95}).

\begin{definition}
	A \textit{parabolic subbundle} $F_*$ of a parabolic bundle $E_*$ is a subbundle $F\subset E$ of the underlying vector bundle endowed with an induced parabolic structure. An induced parabolic structure on $F$ is defined as follows. For every parabolic point $p\in D$, the quasi-parabolic structure on $F$, i.e. the flag in $F_p$ is given by 
	\[
		F_p \eqqcolon F_{p,1}\supsetneq F_{p,2} \supsetneq \dots \supsetneq F_{p,r'_p} \supsetneq \{0\},
	\]
	where $F_{p,i}= F_p \cap E_{p,i}$, i.e. we are considering the intersection with the already given flag in $E_p$, and also scrapping all the repetitions of subspaces in the filtration. The weights $0\leq \alpha'_1(p) < \alpha'_2(p) < \dots < \alpha'_{r'_p}(p) < 1$ are taken to be the largest possible among the given weights which are allowed after the intersections, i.e. 
	\[
	\alpha'_i(p) = \mathrm{max}_j\{\alpha_j(p)| F_p \cap E_{p,j}=F_{p,i} \}=\mathrm{max}_j\{\alpha_j(p)| F_{p,i}\subseteq E_{p,j} \}
	\]
	That is to say, the weight associated to $F_{p,i}$ is the weight $\alpha_j(p)$ such that $F_{p,i}\subseteq E_{p,j}$ but $F_{p,i}\nsubseteq E_{p,j+1}$.
\end{definition}

\begin{definition}
A parabolic bundle $E_*$ is called \textit{semistable} (resp. \textit{stable}) if every nonzero proper subbundle $F_* \subset E_*$ satisfies
\[
	\mu_{\mathrm{par}}(F_*) \leq \mu_{\mathrm{par}}(E_*) \hspace{0.2cm} (\mathrm{resp. } \hspace{0.2cm} <).
\]

\end{definition}

The moduli space $\mathcal{M}(\alpha,r,d)$ of stable parabolic bundles over $X$ of fixed rank $r$ and degree $d$ and parabolic structure $\alpha$ was constructed by Mehta and Seshadri in \cite{MS80}. They also showed that $\mathcal{M}(\alpha,r,d)$ is a normal projective variety of dimension
\[
\dim \mathcal{M}(\alpha,r,d) =r^2(g-1) + 1 + \dfrac{n(r^2-r)}{2}, 
\]
where $n$ is the number of marked points and the last summand comes from the fact that the parabolic structure have full flags over each parabolic point.
\subsection{Parabolic Higgs bundles}
Let $K$ be the canonical bundle over $X$. We write $K(D) \coloneqq K \otimes \mathcal{O}(D)$.

\begin{definition}
A \textit{(strongly) parabolic Higgs bundle} over $X$ is a parabolic bundle $E_*$ over $X$ together with Higgs field $\Phi : E_* \to E_* \otimes K(D)$, such that $\Phi$ is strongly parabolic i.e. $\Phi(E_{p,i}) \subset E_{p,i+1} \otimes \left.K(D)\right|_p$ for all $p \in D$. 
\end{definition}

There is also a notion of parabolic Higgs bundle where the Higgs field $\Phi$ is only assumed to be parabolic, i.e. $\Phi(E_{p,i}) \subset E_{p,i} \otimes \left.K(D)\right|_p$ for all $p \in D$. However in this paper we will always assume that the Higgs field is strongly parabolic. 

\begin{definition}
A subbundle $F_* \subset E_*$ is called $\Phi$\textit{-invariant} if $\Phi(F_*)\subset F_*\otimes K(D)$.
\end{definition}

\begin{definition}
A parabolic Higgs bundle $(E_*,\Phi)$ is called \textit{semistable} (resp. \textit{stable}) if every nonzero proper $\Phi$-invariant subbundle $F_* \subset E_*$ satisfies
\[
  \mu_{\mathrm{par}}(F_*) \leq \mu_{\mathrm{par}}(E_*) \hspace{0.2cm} (\mathrm{resp. } \hspace{0.2cm} <).
\]
\end{definition}

The moduli space $\mathcal{M}_{\mathrm{Higgs}}(\alpha,r,d)$ of stable parabolic Higgs bundles of fixed rank $r$, degree $d$ and parabolic structure $\alpha$ was constructed by Yokogawa in \cite{Y93} (see \cite{BY96} for more details). It is a normal quasi-projective complex variety of dimension
\[
\dim\mathcal{M}_{\mathrm{Higgs}}(\alpha,r,d) = 2\dim \mathcal{M}(\alpha,r,d).
\]

The parabolic version of the Serre duality (see \cite{BY96}, \cite{Y95}) says that 
\[
H^1(\mathrm{PEnd}(E_*)) \cong H^0(\mathrm{SPEnd}(E_*) \otimes K(D))^*.
\]
Therefore, there is an open embedding $T^*\mathcal{M}(\alpha,r,d) \xhookrightarrow{}  \mathcal{M}_{\mathrm{Higgs}}(\alpha,r,d)$ and that is why the dimension of $\mathcal{M}_{\mathrm{Higgs}}(\alpha,r,d)$ is twice the dimension of $\mathcal{M}(\alpha,r,d)$. Thus the moduli space of parabolic Higgs bundles $\mathcal{M}_{\mathrm{Higgs}}(\alpha,r,d)$ has a symplectic structure induced from the natural symplectic structure of the cotangent space.

Let $\mathrm{Jac}^d(X)$ denote the space of degree $d$ line bundles over $X$. Consider the determinant map
\begin{align*}
    \mathrm{det} : \mathcal{M}_{\mathrm{Higgs}}(\alpha,r,d) &\longrightarrow \mathrm{Jac}^d(X) \times H^0(X,K)\\
    (E_*,\Phi) &\longmapsto (\wedge^rE_*,\mathrm{trace}(\Phi)).
\end{align*}
Since $\Phi$ is strongly parabolic, $\mathrm{trace}(\Phi) \in H^0(X,K)\subset H^0(X,K(D))$. The moduli space $\mathcal{M}^{\xi}_{\mathrm{Higgs}}(\alpha,r,d)$ of stable parabolic Higgs bundles with fixed determinant $\xi$ is defined by the fiber $\mathrm{det}^{-1}(\xi,0)$, i.e.
\[
\mathcal{M}^{\xi}_{\mathrm{Higgs}}(\alpha,r,d) \coloneqq \mathrm{det}^{-1}(\xi,0).
\]
If the Higgs field is zero, then the dimension of the moduli space $\mathcal{M}^{\xi}_{\mathrm{Higgs}}(\alpha,r,d)$ is given by
\[
\dim\mathcal{M}^{\xi}_{\mathrm{Higgs}}(\alpha,r,d) = 2(g-1)(r^2-1)+nr(r-1).
\]
\subsection{Spectral correspondence}
We will give a description of the spectral correspondence for the moduli space $\mathcal{M}^{\xi}_{\mathrm{Higgs}}(\alpha,r,d)$ (although a similar description can be given for the moduli space $\mathcal{M}_{\mathrm{Higgs}}(\alpha,r,d)$). 

Let $p : \mathrm{Tot}(K(D)) \to X$ be the natural projection from the total space of $K(D)$ to $X$ and let $x \in H^0(\mathrm{Tot}(K(D)),p^*K)$ denote the tautological section. Since the Higgs field $\Phi$ is strongly parabolic, the residue at every point of $D$ is nilpotent. Therefore the trace of the map $$\wedge^i\Phi : \wedge^iE_* \to \wedge^iE_* \otimes K(D)^i$$ lies in $K^i(D^{i-1})$ for each $2\leq i \leq r$, where $K^i(D^{j})$ denote the tensor product of the $i$-th power of $K$ and the $j$-th power of $\mathcal{O}(D)$. The coefficients of the characteristic polynomial of $\Phi$ are precisely given by $s_i=\mathrm{trace}(\wedge^i\Phi)$. Therefore, we have the $\textit{Hitchin map}$
\begin{align*}
	h_{\mathrm{par}} : \mathcal{M}^{\xi}_{\mathrm{Higgs}}(\alpha,r,d) &\longrightarrow \mathcal{H}_{\mathrm{par}} \coloneqq \bigoplus_{i=2}^{r} H^0(X, K^{i}(D^{i-1}))\\
	(E_*,\Phi) &\longmapsto (s_2,\dots , s_{r}).
\end{align*}
By Riemann-Roch theorem, the dimension of the base $\mathcal{H}_{\mathrm{par}}$ is 
\[
r^2(g-1) +\dfrac{nr(r-1)}{2}, 
\]
which is same as the half the dimension of the moduli space $\mathcal{M}^{\xi}_{\mathrm{Higgs}}(\alpha,r,d)$.

Given $s=(s_2,s_3\dots, s_{r}) \in \mathcal{H}_{\mathrm{par}}$ with 
\[
s_i \in H^0(X, K^{i}(D^{i-1})) \subset H^0(X, K(D)^i),
\]
the \textit{spectral curve} $X_s$ in $\mathrm{Tot}(K(D))$ is defined by
\[
x^{r} + \tilde{s}_2 x^{r-2} + \tilde{s}_3 x^{r-3} \cdots + \tilde{s}_{r}=0
\]
where $\tilde{s}_i = p^*(s_i)$ and $x \in H^0(\mathrm{Tot}(K(D)),p^*K)$ is the tautological section. Let $$\pi : X_s \to X$$ be the restriction of the projection $p$. For a generic point $s \in \mathcal{H}_{\mathrm{par}}$, the spectral curve $X_s$ is smooth and by \cite{GL11} the fiber $h_{\mathrm{par}}^{-1}(s)$ of the Hitchin map is isomorphic to 
\[
\mathrm{Prym}(X_s/X) = \{L \in \mathrm{Pic}(X) : \det \pi_*L \cong \xi\}.
\]

\subsection{Parabolic $\mathrm{GSp}(2m,\mathbb{C})$-Higgs bundles}	
	Let us consider the standard symplectic form
\[
J = \begin{bmatrix} 0 & I_{m} \\ -I_{m} & 0 \end{bmatrix}
\]
on $\mathbb{C}^{2m}$. Then the general symplectic group defined by
\[
\mathrm{GSp}(2m,\mathbb{C}) = \{ A \in \mathrm{GL}(2m,\mathbb{C}) : AJA^t=\lambda_A J \text{  for some  } \lambda_A \in \mathbb{C}^*\}
\]
is an extension of $\mathbb{C}^*$ by the symplectic group $\mathrm{Sp}(2m,\mathbb{C})$, i.e. there exist an exact sequence
\[
1 \to \mathrm{Sp}(2m,\mathbb{C}) \xrightarrow{} \mathrm{GSp}(2m,\mathbb{C}) \xrightarrow[]{p} \mathbb{C}^* \xrightarrow[]{} 1,
\]
where $p(A)=\lambda_A$. Therefore, $\det(A)=p(A)^m=\lambda_A^m$ for every $A \in \mathrm{GSp}(2m,\mathbb{C})$.\\
The Lie algebra of $\mathrm{GSp}(2m,\mathbb{C})$ is given by 
\[
\mathfrak{gsp}(2m,\mathbb{C}) = \{M \in \mathfrak{gl}(2m,\mathbb{C}) :  MJ + JA^t = \frac{\mathrm{tr}(M)}{m}J\} \cong \mathfrak{sp}(2m,\mathbb{C}) \oplus \mathbb{C}.
\]
The decomposition $M = N + \frac{\mathrm{tr}(M)}{m}I_{2m}$, with $N \in \mathfrak{sp}(2m,\mathbb{C})$ produces the above isomorphism. 

Let $L$ be a line bundle over $X$ of degree $l$. Let $E_*$ be a parabolic bundle and let
\begin{equation}\label{bilinear}
\varphi : E_* \otimes E_* \to L
\end{equation}
be a homomorphism of parabolic bundles. The trivial line bundle $\mathcal{O}_X$ equipped with the trivial parabolic structure is realized as a parabolic subbundle of $E_* \otimes E^\vee_*$ by sending a locally defined function $f$ to the locally defined endomorphism of $E$ given by pointwise multiplication with $f$. Let
\begin{equation}\label{nondegenerate}
\tilde{\varphi} : E_* \to L \otimes E^\vee_*
\end{equation}
be the homomorphism defined by the composition
\[
E_* = E_* \otimes \mathcal{O}_X  \xhookrightarrow{} E_* \otimes (E_* \otimes E^\vee_*) = (E_* \otimes E_*) \otimes E^\vee_* \xrightarrow{\varphi \otimes Id} L \otimes E^\vee_*.
\]
\begin{definition}\label{maindefn}
	 A \textit{symplectic parabolic bundle} is a pair $(E_*,\varphi)$ of the above form such that $\varphi$ is anti-symmetric and the homomorphism $\tilde{\varphi}$ is an isomorphism. 
\end{definition}
Suppose $E$ is the underlying vector bundle of a symplectic parabolic bundle $(E_*,\varphi)$. The tensor product $E \otimes E$ is a coherent subsheaf of the vector bundle underlying the parabolic bundle $E_*\otimes E_*$. Therefore, $\varphi$ induces a homomorphism
\begin{equation}\label{underlying}
\hat{\varphi} : E \otimes E \to L
\end{equation}
of vector bundles.
\begin{definition}
A subbundle $F \subset E$ of the underlying bundle of $(E_*,\varphi)$ is called \textit{isotropic} if $\hat{\varphi}(F \otimes F)=0$.
\end{definition}
 
 A parabolic Higgs field $\Phi$ on a symplectic parabolic bundle  $(E_*,\varphi)$ is said to be \textit{compatible} with  $\varphi$ if $\tilde{\varphi}$ takes $\Phi$ to the induced parabolic Higgs field on $L \otimes E^\vee_*$ (we are considering the zero section as the Higgs field on $L$). We can describe this compatibility condition locally. A strongly parabolic Higgs field $\Phi$ on $E_*$ can be viewed as a holomorphic section of $\mathrm{SPEnd}(E_*) \otimes K(D)$. Let $s$ and $t$ be any holomorphic sections of $E_*$ defined over an open subset $U \subset X$. Consider
\[
\hat{\varphi}_{\Phi}(s,t) \coloneqq \hat{\varphi}(\Phi(s)\otimes t)+ \hat{\varphi}(s \otimes \Phi(t)) \in \Gamma(U,L\otimes K(D)),
\]
where $\hat{\varphi}$ is the pairing defined in \ref{underlying}. The Higgs field $\Phi$ is said to be compatible with $\phi$ if and only if $\hat{\varphi}_{\Phi}(s,t)=0$ for all sections $s$ and $t$.
 \begin{definition}
A \textit{symplectic parabolic Higgs bundle} $(E_*,\varphi,\Phi)$ is a symplectic parabolic bundle $(E_*,\varphi)$  equipped with a parabolic Higgs field $\Phi$ on $E_*$ which is compatible with $\varphi$.
\end{definition}

When the parabolic weights are all rational, the notion of symplectic parabolic Higgs bundle is equivalent to the notion of parabolic $\mathrm{GSp}(2m,\mathbb{C})$-Higgs bundles (see \cite{BMW11}).

\subsection{Parabolic $\mathrm{GSO}(2m,\mathbb{C})$-Higgs bundles}
Let $B$ be a non-degenerate symmetric bilinear form on $\mathbb{C}^{2m}$. Also for notational convenience, we denote the corresponding symmetric matrix by $B$. Then the general (even) orthogonal group
\[
\mathrm{GO}(2m,\mathbb{C}) = \{ A \in \mathrm{GL}(2m,\mathbb{C}) : ABA^t=\lambda_A J \text{  for some  } \lambda_A \in \mathbb{C}^*\}.
\]
therefore,, we have $(\det(A))^2=\lambda_A^{2m}$ for every $A \in \mathrm{GO}(2m,\mathbb{C})$.

In this case, there is a \textit{sgn} morphism
\[
\textit{sgn} : \mathrm{GO}(2m,\mathbb{C}) \longrightarrow \{\pm 1\}
\]
sending $A$ to $\det(A)/\lambda_A^m$. The general special orthogonal group is the kernel of this \textit{sgn} morphism and it is denoted by $\mathrm{GSO}(2m,\mathbb{C})= \mathrm{ker}(\textit{sgn})$. So we have a short exact sequence
\[
1 \to \mathrm{GSO}(2m,\mathbb{C}) \xrightarrow{} \mathrm{GO}(2m,\mathbb{C}) \xrightarrow[]{\textit{sgn}} \{\pm 1\} \xrightarrow[]{} 1.
\]

\begin{definition}\label{maindefn}
	 An \textit{orthogonal parabolic bundle} is a pair $(E_*,\varphi)$, where $\varphi$ (as in \ref{bilinear}) is symmetric and the homomorphism $\tilde{\varphi}$ (as in \ref{nondegenerate}) is an isomorphism. 
\end{definition}	 
 \begin{definition}
An \textit{orthogonal parabolic Higgs bundle} $(E_*,\varphi,\Phi)$ is an orthogonal parabolic bundle $(E_*,\varphi)$  equipped with a parabolic Higgs field $\Phi$ on $E_*$ which is compatible with $\varphi$.
\end{definition}
As in the $\mathrm{GSp}$-case, for rational weights the notions of (even) orthogonal parabolic Higgs bundles and parabolic $\mathrm{GSO}(2m,\mathbb{C})$-Higgs bundles coincide.

\subsection{Moduli space}

\begin{definition}
	A symplectic/orthogonal parabolic Higgs bundle $(E_*,\varphi,\Phi)$ is said to be \textit{semistable} (resp. \textit{stable}) if every nonzero isotropic subbundle $F \subset E$ such that $\Phi(F_*) \subset F_* \otimes K(D)$ satisfies
	\[
	\mu_{par}(F_*) \leq \mu_{par}(E_*) \hspace{0.4cm}(\text{resp.} \hspace{0.15cm}  < )
	\]
	holds, where $F_*\subset E_*$ has the induced parabolic structure.
\end{definition}

The moduli space $\mathcal{M}_G(\alpha)$ of stable parabolic $G$-bundles of a fixed topological type and with a fixed parabolic structure $\alpha$ is a normal quasi-projective variety (see \cite{BBN01}, \cite{BR89}) of dimension
\[
\dim\mathcal{M}_G(\alpha)= \dim Z(G)+ (g-1)\dim(G) + n\dim(G/B),
\]
where $Z(G)$ denotes the the center of $G$ and $n$ is the number of parabolic points. The last summand comes from the fact that the flags we are considering over each point of $D$ are full flags and $B$ is the Borel subgroup of $G$ determined by $\alpha$. The moduli space $\mathcal{M}_{G\mathrm{-Higgs}}(\alpha)$ of stable parabolic $G$-Higgs bundles (see \cite{R16}) is a normal quasi-projective variety of dimension 
\[
\dim \mathcal{M}_{G\mathrm{-Higgs}}(\alpha) = 2\dim \mathcal{M}_G(\alpha).
\]
In particular when $G=\mathrm{GSp}(2m,\mathbb{C})$, the moduli space $\mathcal{M}_{\mathrm{GSp-Higgs}}(\alpha,2m,d)$ of stable parabolic $\mathrm{GSp}(2m,\mathbb{C})$-Higgs bundle of fixed degree $d$ has dimension
\[
\dim \mathcal{M}_{\mathrm{GSp-Higgs}}(\alpha,2m,d) = 2m(2m+1)(g-1) + 2m^2n.
\]
Similarly, the moduli space $\mathcal{M}_{\mathrm{GSO-Higgs}}(\alpha,2m,d)$ of stable parabolic $\mathrm{GSO}(2m,\mathbb{C})$-Higgs bundle of fixed degree $d$ has dimension
\[
\dim \mathcal{M}_{\mathrm{GSO-Higgs}}(\alpha,2m,d) = 2m(2m-1)(g-1) + 2mn(m-1).
\]

\subsection*{Notation:} From now on, for notational convenience we shall denote a parabolic bundle $E_*$ by $E$.

\section{Parabolic $\mathrm{PSp}(2m,\mathbb{C})$-Higgs bundles}
In this section we will discuss the Hitchin fibration for the moduli space of parabolic $\mathrm{PSp}(2m,\mathbb{C})$-Higgs bundles. The projective symplectic group $\mathrm{PSp}(2m,\mathbb{C})$ is given by the following exact sequence:
\begin{equation*}
1 \xrightarrow{} \mathbb{C}^* \xrightarrow{c \to cI_{2m}} \mathrm{GSp}(2m,\mathbb{C}) \xrightarrow{} \mathrm{PSp}(2m,\mathbb{C})  \xrightarrow{} 1.
\end{equation*}
The sheaf version of this sequence induces the following exact sequence in homology :
\begin{equation}
   H^1(X,\mathcal{O}_X^*) \xrightarrow{} H^1(X,\mathrm{GSp}(2m,\mathcal{O}_X)) \xrightarrow{q} H^1(X,\mathrm{PSp}(2m,\mathcal{O}_X)) \to 0
\end{equation}
The surjectivity of the map $q$ implies that there is a bijective correspondence between the parabolic $\mathrm{PSp}(2m,\mathbb{C})$-bundles and the equivalence classes of parabolic $\mathrm{GSp}(2m,\mathbb{C})$-bundles with respect to the action given by the tensor product of line bundles on the associated bundles. If $V$ is a parabolic $\mathrm{PSp}(2m,\mathbb{C})$-bundle and $\Tilde{V}$ is a parabolic $\mathrm{GSp}(2m,\mathbb{C})$-bundle such that $q(\Tilde{V}) = V$, then we call that $\Tilde{V}$ is a lifting of $V$ to a parabolic $\mathrm{GSp}(2m,\mathbb{C})$-bundle.

The group $\mathrm{PSp}(2m,\mathbb{C})$ can also be defined by the quotient of the symplectic group $\mathrm{Sp}(2m,\mathbb{C})$ by the action of a finite group by the following exact sequence :
\begin{equation*}
1 \xrightarrow{} \{\pm 1\} \xrightarrow{1 \to I_{2m}} \mathrm{Sp}(2m,\mathbb{C}) \xrightarrow{}  \mathrm{PSp}(2m,\mathbb{C}) \xrightarrow{} 1.
\end{equation*}
Since it is a quotient by a finite group, the Lie algebras are equal, i.e.
\begin{equation}\label{equalLie}
\mathfrak{sp}(2m,\mathbb{C}) = \mathfrak{psp}(2m,\mathbb{C}).
\end{equation}

Consider a parabolic $\mathrm{PSp}(2m,\mathbb{C})$-Higgs bundle $(V,\eta)$ which lifts to a parabolic $\mathrm{GSp}(2m,\mathbb{C})$-Higgs bundle $(\Tilde{V},\Tilde{\eta})$, i.e. $q(\Tilde{V},\Tilde{\eta}) = (V,\eta)$. Let $(E,\Phi,\varphi,L)$ be the bundle corresponding to $(\Tilde{V},\Tilde{\eta})$. Then $(V,\eta)$ corresponds to the equivalence class $[(E,\Phi,\varphi,L)]$ where the equivalence relation $\sim_{\mathrm{Jac}(X)}$ is given by :
\[
(E,\Phi,\varphi,L) \sim_{\mathrm{Jac}(X)} (E\otimes M,\Phi \otimes \mathrm{Id}_M,\varphi_M,L\otimes M^2) \hspace{1cm} \mathrm{for} \hspace{0.1cm} \mathrm{any} \hspace{0.2cm} M\in {\mathrm{Jac}(X)}
\]
where 
\[
\varphi_M : (E\otimes M) \otimes (E\otimes M) \to L \otimes M^2
\]
is the induced symplectic form on $E\otimes M$ taking values in $L\otimes M^2$. 


Let $M_{\mathrm{GSp-Higgs}}(\alpha,2m) = \coprod_{d\in \mathbb{Z}} \mathcal{M}_{\mathrm{GSp-Higgs}}(\alpha,2m,d)$ denote the moduli stack of parabolic $\mathrm{GSp}$-Higgs bundles of fixed rank $2m$ with any degree. Then the $\mathrm{Jac}(X)$-orbit $[(E,\Phi,\varphi,L)]$ of $(E,\Phi,\varphi,L)$ under the action of $\mathrm{Jac}(X)$ on the stack $M_{\mathrm{GSp-Higgs}}(\alpha,2m)$ is defined by:
\begin{align*}
    M_{\mathrm{GSp-Higgs}}(\alpha,2m) \times \mathrm{Jac}(X) &\longrightarrow M_{\mathrm{GSp-Higgs}}(\alpha,2m) \\
    ((E,\Phi,\varphi,L),M) &\mapsto (E\otimes M,\Phi \otimes \mathrm{Id}_M,\varphi_M,L\otimes M^2).
\end{align*}

Note that $\deg(L\otimes M^2) = \deg(L)+2\deg(M)$, i.e. it changes the degree of $L$ by a multiple of $2$. For a parabolic $\mathrm{PSp}(2m,\mathbb{C})$- Higgs bundle $[(E,\Phi,\varphi,L)]$, the isomorphism $E \cong E^\vee \otimes L$ implies that $\mathrm{pardeg}(E) = m\deg(L)$. So the parabolic degree of $E$ is determined by the degree of $L$. Therefore, 
\begin{align*}
\mathrm{pardeg}(E\otimes M) &= m\deg(L\otimes M^2) \\ &=m\deg(L) + 2m\deg(M)\\ &= \mathrm{pardeg}(E) + 2m\deg(M).
\end{align*}
Therefore, the parabolic degree of a parabolic $\mathrm{PSp}(2m,\mathbb{C})$-Higgs bundle can be defined as follows :

\begin{definition}
Let $(V,\eta)$ be a parabolic $\mathrm{PSp}(2m,\mathbb{C})$-Higgs bundle and let $(\Tilde{V},\Tilde{\eta})$ be a lifting to a parabolic $\mathrm{GSp}(2m,\mathbb{C})$-Higgs bundle. Let $(E,\Phi,\varphi,L)$ be the datum corresponding to $(\Tilde{V},\Tilde{\eta})$ with parabolic degree $ml$, where $l=\deg(L)$. Then the $\textit{parabolic degree}$ of $(V,\eta)$ is given by the class $\overline{ml} \in \mathbb{Z}/m\mathbb{Z}$.
\end{definition}

Threfore, we will consider the moduli space $\mathcal{M}_{\mathrm{PSp-Higgs}}(\alpha,2m,\overline{ml})$ of stable parabolic $\mathrm{PSp}(2m,\mathbb{C})$-Higgs bundles of fixed rank $2m$ and parabolic degree $\overline{ml}$.

Consider a basis $\{s_{2i}\}_{i =1,\dots , m}$ of invariant polynomials of the lie algebra $\mathfrak{sp}(2m,\mathbb{C}) = \mathfrak{psp}(2m,\mathbb{C})$, where $s_{2i}=\mathrm{tr}(\wedge^{2i}\eta)$. Therefore the Hitchin morphism for the moduli space of parabolic $\mathrm{PSp}(2m,\mathbb{C})$-Higgs bundles is of the form
\begin{align*}
h_{\mathrm{PSp-par}} : \mathcal{M}_{\mathrm{PSp-Higgs}}(\alpha,2m,\overline{ml}) &\longrightarrow \mathcal{H}_{\mathrm{PSp-par}} \coloneqq \bigoplus_{i=1}^{m} H^0(X, K^{2i}(D^{2i-1}))\\
	(V,\eta) &\longmapsto (s_2,\dots , s_{2m}).
\end{align*}
Observe that if $(\Tilde{V},\Tilde{\eta})$ is a lifting of $(V,\eta)$ to a parabolic $\mathrm{GSp}(2m,\mathbb{C})$-Higgs bundle then 
\[
h_{\mathrm{GSp-par}}(\Tilde{V},\Tilde{\eta}) = h_{\mathrm{PSp-par}}(V,\eta).
\]
Let $s=(s_2,s_4,\dots , s_{2m}) \in \mathcal{H}_{\mathrm{PSp-par}}$ be a generic point of the Hitchin base. The spectral curve 
\[
\pi : X_s \to X
\]
is defined by the equation
\[
x^{2m} + s_2x^{2m-2} + s_4x^{2m-4}+ \cdots + s_{2m} = 0.
\]
For a generic $s\in \mathcal{H}_{\mathrm{PSp-par}}$, the corresponding spectral curve $X_s$ is smooth (see \cite{BNR89}).
Since all odd coefficients of the above equation are zero, the spectral curve $X_s$ possesses an involution $\sigma : X_s \to X_s$ defined by $\sigma(\lambda)=-\lambda$. Therefore, we can define a $2$-fold covering map
\[
q : X_s \to X_s/\sigma.
\]
Since $\sigma$ sends a degree zero line bundle on $X_s$ to a degree zero line bundle, it acts on the Jacobian $\text{Jac}(X_s)$. The \textit{Prym variety} $P_{s,\sigma}\coloneqq \mathrm{Prym}(X_s, X_s/\sigma)$ is given by 
	\[
	P_{s,\sigma} \coloneqq \mathrm{Prym}(X_s, X_s/\sigma) = \{N \in \text{Jac}(X_s): \sigma^*N \cong N^\vee\}
	\]
	and it is of dimension 
	\[
	\dim P_{s,\sigma}=g(X_s) - g(X_s/\sigma).
	\]
Following \cite[Theorem 4.1]{R20},  we can give a different description of the Prym variety. Let $J \in P_{s,\sigma}$ be an element in the Prym variety. Consider the line bundle 
\[
U=J \otimes R^\vee,
\]
where $R = (K_{X_s} \otimes \pi^*K^\vee \otimes \pi^*L^\vee)^{1/2}$ is a holomorphic square root. Then it follows that $U$ satisfies the isomorphism
\begin{equation}\label{alternative}
\sigma^*U \cong U^\vee \otimes (K_{X_s} \otimes \pi^*K^\vee)^{-1} \otimes \pi^*L.
\end{equation}

Similarly, let $U \in \text{Jac}(X_s)$ be a line bundle satisfying the equation (\ref{alternative}). Then the line bundle $J = U \otimes R \in P_{s,\sigma}$ is an element in the Prym variety. Therefore, there is a bijective correspondence between the Prym variety $P_{s,\sigma}$ and 
\begin{equation}\label{prym}
    \Omega_{s,\sigma} \coloneqq \{(U,L,\tau)\hspace{0.1cm}|\hspace{0.1cm} U \in \text{Jac}(X_s), \tau : \sigma^*U \cong U^\vee \otimes (K_{X_s} \otimes \pi^*K^\vee)^{-1} \otimes \pi^*L \}.
\end{equation}

From now on, we will refer an element in the Prym variety as an element of $\Omega_{s,\sigma}$.

The Jacobian $\mathrm{Jac}(X)$ acts on $\Omega_{s,\sigma}$ as follows:\ 
\begin{align*}
    \Omega_{s,\sigma} \times \mathrm{Jac}(X) &\longrightarrow \Omega_{s,\sigma}\\
    ((U,L,\tau),M) &\longmapsto (U \otimes \pi^*M, L \otimes M^2, \tau_M) 
\end{align*}
where $\tau_M= \tau \otimes \text{Id}_{\pi^*M}$ is the following isomorphism
\begin{align*}
    \sigma^*(U \otimes \pi^*M) &\cong \sigma^*U \otimes \pi^*M\\
    &\cong U^\vee \otimes (K_{X_s} \otimes \pi^*K^\vee)^{-1} \otimes \pi^*L \otimes \pi^*M\\
    &\cong U^\vee \otimes \pi^*M^\vee \otimes \pi^*M \otimes (K_{X_s} \otimes \pi^*K^\vee)^{-1} \otimes \pi^*L \otimes \pi^*M\\
    &\cong (U\otimes \pi^*M)^\vee \otimes (K_{X_s} \otimes \pi^*K^\vee)^{-1} \otimes \pi^*(L\otimes M^2).
\end{align*}
\begin{theorem}\label{PSp}
For a generic point $s\in \mathcal{H}_{\mathrm{PSp-par}}$, the fiber $h_{\mathrm{PSp-par}}^{-1}(s)$ is isomorphic to the quotient $\Omega_{s,\sigma}/\mathrm{Jac}(X)$.
\end{theorem}
\begin{proof}
Let $(V,\eta) \in h_{\mathrm{PSp-par}}^{-1}(s)$ be a parabolic $\mathrm{PSp}(2m,\mathbb{C})$-Higgs bundle in the fiber of the Hitchin morphism and let $(\Tilde{V},\Tilde{\eta})$ be a lifting to a parabolic $\mathrm{GSp}(2m,\mathbb{C})$-Higgs bundle which corresponds to the datum of $(E,\Phi,\varphi,L)$. Then $(V,\eta)$ corresponds to the datum of the $\mathrm{Jac}(X)$-orbit $[(E,\Phi,\varphi,L)]$ uniquely. By \cite[Theorem 4.1]{R20}, the datum of $(E,\Phi,\varphi,L)$ corresponds to an element of the Pyrm variety $\Omega_{s,\sigma}$ via the spectral correspondence. Consider an element $M \in \mathrm{Jac}(X)$, i.e. a degree zero line bundle over $X$. Then by the projection formula the datum of $(E \otimes M,\Phi \otimes \text{Id}_M,\varphi_M,L\otimes M^2)$ corresponds uniquely to an element of $\Omega_{s,\sigma}/\mathrm{Jac}(X)$. Therefore, we conclude that the datum of the  $\mathrm{Jac}(X)$-orbit $[(E,\Phi,\varphi,L)]$ corresponds uniquely to the datum of the $\mathrm{Jac}(X)$-orbit of an element of $\Omega_{s,\sigma}$ via the spectral correspondence.
\end{proof}

\section{Parabolic $\mathrm{PSO}(2m,\mathbb{C})$-Higgs bundles}\label{PSO}
As in the previous section, there is a bijective correspondence between the parabolic $\mathrm{PSO}(2m,\mathbb{C})$-bundles and the equivalence classes of parabolic $\mathrm{GSO}(2m,\mathbb{C})$-bundles. So, for  every parabolic $\mathrm{PSO}(2m,\mathbb{C})$-bundle $V$ there is a lifting $\Tilde{V}$ of $V$ to a parabolic $\mathrm{GSO}(2m,\mathbb{C})$-bundle.

The group $\mathrm{PSO}(2m,\mathbb{C})$ can be defined by the quotient of $\mathrm{SO}(2m,\mathbb{C})$ by the action of a finite group by the following exact sequence :
\begin{equation*}
1 \xrightarrow{} \{\pm 1\} \xrightarrow{1 \to I_{2m}} \mathrm{SO}(2m,\mathbb{C}) \xrightarrow{}  \mathrm{PSO}(2m,\mathbb{C}) \xrightarrow{} 1.
\end{equation*}
Therefore, we have
\begin{equation}\label{equalLie}
\mathfrak{so}(2m,\mathbb{C}) = \mathfrak{pso}(2m,\mathbb{C}).
\end{equation}

Let $(V,\eta)$ be a parabolic $\mathrm{PSO}(2m,\mathbb{C})$-Higgs bundle lifting to the parabolic $\mathrm{GSO}(2m,\mathbb{C})$-Higgs bundle $(\Tilde{V},\Tilde{\eta})$ and let $(E,\Phi,\varphi,L)$ be the bundle corresponding to $(\Tilde{V},\Tilde{\eta})$. Then $(V,\eta)$ corresponds to the equivalence class $[(E,\Phi,\varphi,L)]$ where the equivalence relation $\sim_{\mathrm{Jac}(X)}$ is given by :
\[
(E,\Phi,\varphi,L) \sim_{\mathrm{Jac}(X)} (E\otimes M,\Phi \otimes \mathrm{Id}_M,\varphi_M,L\otimes M^2) \hspace{1cm} \mathrm{for} \hspace{0.1cm} \mathrm{any} \hspace{0.2cm} M\in {\mathrm{Jac}(X)}
\]
where 
\[
\varphi_M : (E\otimes M) \otimes (E\otimes M) \to L \otimes M^2
\]
is the induced symmetric bilinear nondegenerate form on $E\otimes M$ taking values in $L\otimes M^2$. 

As in the previous case, the parabolic degree of a parabolic $\mathrm{PSO}(2m,\mathbb{C})$-Higgs bundle can be defined as follows:

\begin{definition}
Let $(V,\eta)$ be a parabolic $\mathrm{PSO}(2m,\mathbb{C})$-Higgs bundle and let $(\Tilde{V},\Tilde{\eta})$ be a lifting to a parabolic $\mathrm{GSO}(2m,\mathbb{C})$-Higgs bundle. Let $(E,\Phi,\varphi,L)$ be the datum corresponding to $(\Tilde{V},\Tilde{\eta})$ with parabolic degree $ml$, where $l=\deg(L)$. Then the $\textit{parabolic degree}$ of $(V,\eta)$ is given by the class $\overline{ml} \in \mathbb{Z}/m\mathbb{Z}$.
\end{definition}

Let $\{s_{2i}\}_{i =1,\dots , m}$ be a basis of invariant polynomials of the lie algebra $\mathfrak{so}(2m,\mathbb{C}) = \mathfrak{pso}(2m,\mathbb{C})$
In this case, the coefficient $s_{2m}$ is a square of a polynomial $p_m\in H^0(X,K(D)^m)$, the Pfaffian, of degree $m$. A basis for the invariant polynomials on the Lie algebra $\mathfrak{so}(2m,\mathbb{C}) = \mathfrak{pso}(2m,\mathbb{C})$ is given by the coefficients $\{s_2, ..., s_{2m-2},p_m\}$. Therefore, the Hitchin map for the moduli space of parabolic $\mathrm{PSO}(2m,\mathbb{C})$-Higgs bundles is given by
\begin{align*}
h_{\mathrm{PSO-par}} : \mathcal{M}_{\mathrm{PSO-Higgs}}(\alpha,2m,\overline{ml}) &\longrightarrow \mathcal{H}_{\mathrm{PSO-par}} \coloneqq \bigoplus_{i=1}^{m-1} H^0(X, K^{2i}(D^{2i-1})) \oplus H^0(X,K(D)^{m})\\
	(V,\eta) &\longmapsto (s_2,\dots , s_{2m-2},p_m).
\end{align*} 
For $s=(s_2,\dots, s_{2m-2},p_m)\in \mathcal{H}_{\mathrm{PSO-par}}$, the corresponding spectral curve $X_s$ is given by the equation
\[
x^{2m} + s_2x^{2m-2} + \cdots + s_{2m-2}x^2 + p_m^2 = 0.
\]
The zeroes of $p_m$ are singularities of $X_s$ and these are the only singularities. Since $p_m \in H^0(X,K(D)^{m})$, there are $K(D)^m = m(2g-2+n)$ many singularities. As in the previous case, $X_s$ possesses an involution $\sigma(\eta)=-\eta$. Then the fixed points of this involution $\sigma$ are exactly the singularities of $X_s$. Let $\hat{X_s}$ denote the desingularisation of $X_s$ with genus
\[
g(\hat{X_s})= g(X_s) - \mathrm{number \hspace{0.1cm }of\hspace{0.1cm} singularities}.
\]
Since the singularities of $X_s$ are double points, the involution $\sigma$ on $X_s$ extends to an involution $\hat{\sigma}$ on $\hat{X_s}$. Therefore the \textit{Prym variety} $P_{s,\hat{\sigma}}\coloneqq \mathrm{Prym}(\hat{X_s}, \hat{X_s}/\hat{\sigma})$ is given by 
	\begin{equation}\label{desingular}
	P_{s,\hat{\sigma}} \coloneqq \mathrm{Prym}(\hat{X_s}, \hat{X_s}/\hat{\sigma}) = \{N \in \text{Jac}(\hat{X_s}): \hat{\sigma}^*N \cong N^\vee\}
	\end{equation}
As in the symplectic case, there is a bijective correspondence between the Prym variety $P_{s,\hat{\sigma}}$ and 
\[
\Omega_{s,\hat{\sigma}} \coloneqq \{(U,L,\tau)\hspace{0.1cm}|\hspace{0.1cm} U \in \text{Jac}(\hat{X_s}), \tau : \hat{\sigma}^*U \cong U^\vee \otimes (K_{\hat{X_s}} \otimes \pi^*K^\vee)^{-1} \otimes \pi^*L \}.
\]	
Again the action of the group $\mathrm{Jac}(X)$ on $\Omega_{s,\hat{\sigma}}$ is given by\ 
\begin{align*}
    \Omega_{s,\hat{\sigma}} \times \mathrm{Jac}(X) &\longrightarrow \Omega_{s,\hat{\sigma}}\\
    ((U,L,\tau),M) &\longmapsto (U \otimes \pi^*M, L \otimes M^2, \tau_M) 
\end{align*}
where $\tau_M= \tau \otimes \text{Id}_{\pi^*M}$ is the isomorphism
\[
\hat{\sigma}^*(U \otimes \pi^*M) \cong (U\otimes \pi^*M)^\vee \otimes (K_{\hat{X_s}} \otimes \pi^*K^\vee)^{-1} \otimes \pi^*(L\otimes M^2).
\]
\begin{theorem}\label{PSO_theorem}
For a generic point $s\in \mathcal{H}_{\mathrm{PSO-par}}$, the fiber $h_{\mathrm{PSO-par}}^{-1}(s)$ is isomorphic to $\Omega_{s,\hat{\sigma}}/\mathrm{Jac}(X)$.
\end{theorem}
\begin{proof}
Let $(V,\eta) \in h_{\mathrm{PSp-par}}^{-1}(s)$ lifts to a parabolic $\mathrm{GSO(2m,\mathbb{C})}$-Higgs bundle $(\Tilde{V},\Tilde{\eta})$ whose corresponding datum is $(E,\Phi,\varphi,L)$. Then $(V,\eta)$ corresponds to the datum of the $\mathrm{Jac}(X)$-orbit $[(E,\Phi,\varphi,L)]$ uniquely. By \cite[Theorem 4.2]{R20}, the datum of $(E,\Phi,\varphi,L)$ corresponds to an element of the Pyrm variety $\Omega_{s,\hat{\sigma}}$. Let $M \in \mathrm{Jac}(X)$. Then the datum of $(E \otimes M,\Phi \otimes \text{Id}_M,\varphi_M,L\otimes M^2)$ corresponds uniquely to an element of $\Omega_{s,\hat{\sigma}}/\mathrm{Jac}(X)$. Therefore, we conclude that the datum of the  $\mathrm{Jac}(X)$-orbit $[(E,\Phi,\varphi,L)]$ corresponds uniquely to the datum of the $\mathrm{Jac}(X)$-orbit of an element of $\Omega_{s,\hat{\sigma}}$.

\end{proof}

\section{Fixed line bundle : $\mathcal{O}_X$}
In this section, we will assume that the symplectic/orthogonal form in \ref{bilinear} takes values in the trivial line bundle $\mathcal{O}_X$. In particular, we consider the moduli space of parabolic symplectic/orthogonal Higgs bundles with fixed rank, degree and fixed line bundle $\mathcal{O}_X$. In other words, we are considering the moduli spaces with trivial determinant.

\subsection{Parabolic $\mathrm{PSp}(2m,\mathbb{C})$-Higgs bundles with fixed line bundle $\mathcal{O}_X$}

Let $\mathrm{Jac}_2(X)$ be the subgorup of the Jacobian $\mathrm{Jac}(X)$ which contains the $2$-torsion elements of $\mathrm{Jac}(X)$, i.e.
\[
\mathrm{Jac}_2(X) \coloneqq \{M \in \mathrm{Jac}(X) \hspace{0.1cm} | \hspace{0.1cm} M^2 \cong \mathcal{O}_X \}.
\]
Let $(E,\Phi,\varphi,\mathcal{O}_X)$ be a parabolic symplectic Higgs bundle with the symplectic form $\varphi : E \otimes E \to \mathcal{O}_X$ taking values in $\mathcal{O}_X$  and let $M \in \mathrm{Jac}_2(X)$. Then
\[
\varphi_M : (E \otimes M) \otimes (E\otimes M) \longrightarrow \mathcal{O}_X \otimes M^2 \cong \mathcal{O}_X \otimes \mathcal{O}_X \cong \mathcal{O}_X
\]
defines a symplectic form on $E \otimes M$ with values in $\mathcal{O}_X$.

Since the symplectic form takes values in $\mathcal{O}_X$, a parabolic $\mathrm{PSp}(2m,\mathbb{C})$-Higgs bundle $(V,\eta)$ lifts to a parabolic $\mathrm{Sp}(2m,\mathbb{C})$-Higgs bundle $(\Tilde{V},\Tilde{\eta})$. Let $(E,\Phi,\varphi,\mathcal{O}_X)$ be the parabolic symplectic Higgs bundle corresponding to $(\Tilde{V},\Tilde{\eta})$. Then $(V,\eta)$ corresponds to the equivalence class $[(E,\Phi,\varphi,\mathcal{O}_X)]$ where the equivalence relation $\sim_{\mathrm{Jac}_2(X)}$ is given by :
\[
(E,\Phi,\varphi,\mathcal{O}_X) \sim_{\mathrm{Jac}_2(X)} (E\otimes M,\Phi \otimes \mathrm{Id}_M,\varphi_M,\mathcal{O}_X) \hspace{1cm} \mathrm{for} \hspace{0.1cm} \mathrm{any} \hspace{0.2cm} M\in {\mathrm{Jac}_2(X)}.
\]
As in \ref{prym}, the Prym variety is given by
\[
\Omega_{s,\sigma} \coloneqq \{(U,\mathcal{O}_X,\tau)\hspace{0.1cm}|\hspace{0.1cm} U \in \text{Jac}(X_s), \tau : \sigma^*U \cong U^\vee \otimes (K_{X_s} \otimes \pi^*K^\vee)^{-1} \otimes \pi^*\mathcal{O}_X\}.
\]
Also, the group $\mathrm{Jac}_2(X)$ acts on $\Omega_{s,\sigma}$ by
\begin{align*}
    \Omega_{s,\sigma} \times \mathrm{Jac}_2(X) &\longrightarrow \Omega_{s,\sigma}\\
    ((U,\mathcal{O}_X,\tau),M) &\longmapsto (U \otimes \pi^*M,  \mathcal{O}_X\otimes M^2, \tau_M) \cong (U \otimes \pi^*M,  \mathcal{O}_X, \tau_M)
\end{align*}
where $\tau_M= \tau \otimes \text{Id}_{\pi^*M}$ is the isomorphism
\[
\sigma^*(U \otimes \pi^*M) \cong (U\otimes \pi^*M)^\vee \otimes (K_{X_s} \otimes \pi^*K^\vee)^{-1} \otimes \pi^*\mathcal{O}_X.
\]
\begin{theorem}\label{trivial_PSp}
For a generic point $s\in \mathcal{H}_{\mathrm{PSp-par},\mathcal{O}_X}$ of the parabolic $\mathrm{PSp}$-Hitchin map with the fixed line bundle $\mathcal{O}_X$, the fiber $h_{\mathrm{PSp-par},\mathcal{O}_X}^{-1}(s)$ is isomorphic to the quotient $\Omega_{s,\sigma}/\mathrm{Jac}_2(X)$.
\end{theorem}
\begin{proof}
The proof is similar to the proof of the Theorem \ref{PSp}.
\end{proof}

\subsection{Parabolic $\mathrm{PSO}(2m,\mathbb{C})$-Higgs bundles with fixed line bundle $\mathcal{O}_X$}
As in the symplectic case, a parabolic $\mathrm{PSO}(2m,\mathbb{C})$-Higgs bundle $(V,\eta)$ with the orthogonal form taking values in $\mathcal{O}_X$ lifts to a parabolic $\mathrm{SO}(2m,\mathbb{C})$-Higgs bundle $(\Tilde{V},\Tilde{\eta})$. Let $(E,\Phi,\varphi,\mathcal{O}_X)$ be the parabolic (even) orthogonal Higgs bundle corresponding to $(\Tilde{V},\Tilde{\eta})$. Then $(V,\eta)$ corresponds to the equivalence class $[(E,\Phi,\varphi,\mathcal{O}_X)]$ where the equivalence relation $\sim_{\mathrm{Jac}_2(X)}$ is given by:
\[
(E,\Phi,\varphi,\mathcal{O}_X) \sim_{\mathrm{Jac}_2(X)} (E\otimes M,\Phi \otimes \mathrm{Id}_M,\varphi_M,\mathcal{O}_X) \hspace{1cm} \mathrm{for} \hspace{0.1cm} \mathrm{any} \hspace{0.2cm} M\in {\mathrm{Jac}_2(X)}.
\]
Following the above section \ref{PSO}, the alternative description of the Prym variety (\ref{desingular}) is given by
\[
\Omega_{s,\hat{\sigma}} \coloneqq \{(U,\mathcal{O}_X,\tau)\hspace{0.1cm}|\hspace{0.1cm} U \in \text{Jac}(\hat{X_s}), \tau : \hat{\sigma}^*U \cong U^\vee \otimes (K_{\hat{X_s}} \otimes \pi^*K^\vee)^{-1} \otimes \pi^*\mathcal{O}_X \}.
\]
Also, the action of the subgroup $\mathrm{Jac}_2(X)$ on $\Omega_{s,\hat{\sigma}}$ is given by
\begin{align*}
    \Omega_{s,\hat{\sigma}} \times \mathrm{Jac}_2(X) &\longrightarrow \Omega_{s,\hat{\sigma}}\\
    ((U,\mathcal{O}_X,\tau),M) &\longmapsto (U \otimes \pi^*M, \mathcal{O}_X, \tau_M) 
\end{align*}
where $\tau_M= \tau \otimes \text{Id}_{\pi^*M}$ is the isomorphism
\[
\hat{\sigma}^*(U \otimes \pi^*M) \cong (U\otimes \pi^*M)^\vee \otimes (K_{\hat{X_s}} \otimes \pi^*K^\vee)^{-1} \otimes \pi^*\mathcal{O}_X.
\]
\begin{theorem}\label{trivial_PSO}
For a generic point $s\in \mathcal{H}_{\mathrm{PSO-par},\mathcal{O}_X}$ of the parabolic $\mathrm{PSO}$-Hitchin map with the fixed line bundle $\mathcal{O}_X$, the fiber $h_{\mathrm{PSO-par},\mathcal{O}_X}^{-1}(s)$ is isomorphic to the quotient $\Omega_{s,\hat{\sigma}}/\mathrm{Jac}_2(X)$.
\end{theorem}
\begin{proof}
The proof is similar to the proof of the Theorem \ref{PSO_theorem}.
\end{proof}
\begin{remark}
We can actually consider any degree zero line bundle in place of $\mathcal{O}_X$.
\end{remark}

\section*{Acknowledgement}
This work was supported by the Institute for Basic Science (IBS-R003-D1).

\section*{Data Availability}
Data sharing is not applicable to this article as no new data were created or analyzed in this study.


\begin{thebibliography}{20}

	\bibitem{MS80}
	V.B. Mehta and C.S. Seshadri,
	\emph{Moduli of vector bundles on curves with parabolic structures}, Math. Ann., 248(3) (1980):205–239.
	
	\bibitem{H87a}
	N.J. Hitchin, 
	\emph{The self-duality equations on a Riemann surface}, Proc. LMS 55, 3 (1987), 59-126.
		
	\bibitem{H87}
	N.J. Hitchin,
	\emph{Stable bundles and integrable systems},
		Duke Math. J., Volume 54, Number 1 (1987), 91-114.
		
	\bibitem{H07}
	N.J. Hitchin, 
		\emph{Langlands duality and G2 spectral curves}, Q.J. Math., 58 (2007), 319-344
		
		\bibitem{BR89}
		U. Bhosle, A. Ramanathan,
		\emph{Moduli of parabolic $G$-bundles on curves},
		Math. Z. 202, no. 2 (1989), 161–180.	
	
	\bibitem{M94}
		E. Markman,
		\emph{Spectral curves and integrable systems},
		Compositio Mathematica, tome 93, no 3 (1994), 255-290.
	
	\bibitem{BY96}
	H. Boden and K. Yokogawa, 
	\emph{Moduli spaces of parabolic Higgs bundles and parabolic K(D) pairs over smooth curves}. I, Internat. J. Math. 7 (1996), no. 5, 573–598.
	
		\bibitem{BBN01}
		V. Balaji, I. Biswas and D.S. Nagaraj,
		\emph{Principal bundles over projective manifolds with parabolic
			structure over a divisor},
		Tohoku Math. Jour. 53 (2001), 337–367.
		
		\bibitem{BNR89}
		A. Beauville, M.S. Narasimhan, S. Ramanan,
		\emph{Spectral curves and the generalised theta divisor}. J. Reine Angew. Math. 398,169–179 (1989)
	
	\bibitem{GL11}
	T. Gómez and M. Logares, 
	\emph{A Torelli theorem for the moduli space of parabolic Higgs bundles}, Adv. Geom. 11 (2011) 429–444
	
	
	\bibitem{Y93}
	K. Yokogawa, 
	\emph{Compactification of moduli of parabolic sheaves and moduli of parabolic Higgs sheaves}, J. Math. Kyoto Univ. 33 (1993), no. 2, 451–504.

	\bibitem{Y95}
	K. Yokogawa,
	\emph{Inﬁnitesimal deformation of parabolic Higgs sheaves}, Internat. J. Math. 6 (1995), 125–148.
	
	\bibitem{BMW11}
		I. Biswas, S. Majumder, M.L. Wong,
		\emph{Orthogonal and symplectic parabolic bundles},
		J. Geom. Phys. 61 (2011), 1462–1475.
		
			\bibitem{R16}
		P. Reisert,
		\emph{Moduli spaces of parabolic twisted generalized Higgs bundles (Doctoral dissertation, LMU München)} (2016). Retrieved from https://edoc.ub.uni-muenchen.de/19890/
	
	\bibitem{R20}
	S. Roy,
	\emph{Hitchin Fibration on Moduli of Symplectic and Orthogonal Parabolic Higgs Bundles}, Math Phys Anal Geom 23, 41 (2020)

\end{thebibliography}
\end{document}